\documentclass{amsart}

\usepackage{amssymb}
\usepackage{graphicx,color}
\usepackage{amsmath}
\usepackage{url}
\usepackage[colorlinks=true]{hyperref}

\usepackage{cleveref}
\Crefname{counterexample}{counterexample}{counterexamples}
\Crefname{counterexample}{Counterexample}{Counterexamples}
\Crefname{conjecture}{conjecture}{conjectures}
\Crefname{conjecture}{Conjecture}{Conjectures}

\usepackage{tikz-cd}
\usepackage{tikz}
\usepackage{tkz-euclide}
\usetikzlibrary{matrix}

\newcommand{\limites}{}

\theoremstyle{plain}
\newtheorem{theorem}{Theorem}[section]
\newtheorem{proposition}[theorem]{Proposition}
\newtheorem{lemma}[theorem]{Lemma}
\newtheorem{corollary}[theorem]{Corollary}

\theoremstyle{definition}

\newtheorem{example}[theorem]{Example}

\newtheorem{question}[theorem]{Question}
\newtheorem{conjecture}[theorem]{Conjecture}
\newtheorem{remark}[theorem]{Remark}

\newcommand{\N}{ \ensuremath{\mathbb{N}}}
\newcommand{\Z}{ \ensuremath{\mathbb{Z}}}

\newcommand{\R}{ \ensuremath{\mathbb{R}}}

\renewcommand{\int}{\operatorname{int}}
\newcommand{\width}{\operatorname{width}}
\newcommand{\length}{\operatorname{length}}

\newcommand{\cone}{\ensuremath{\mathrm{cone}}\hspace{1pt}}
\newcommand{\conv}{\ensuremath{\mathrm{conv}}\hspace{1pt}}

\newcommand{\one}{{-\bf 1}}

\usepackage[colorinlistoftodos]{todonotes}

\begin{document}

\title{Hollow polytopes of large width}

\author[G.~Codenotti]{Giulia Codenotti}
\author[F.~Santos]{Francisco Santos}

\address[G.~Codenotti]
{
Institut f\"ur Mathematik, Freie Universit\"at Berlin, Germany
}
\email{giulia.codenotti@fu-berlin.de}

\address[F.~Santos]
{
Dept.~of Mathematics, Statistics and Comp.~Sci., Univ.~of Cantabria, Spain
}
\email{francisco.santos@unican.es}

\thanks{The authors were supported by the Einstein Foundation Berlin under grant EVF-2015-230 and, while they were in residence at the Mathematical Sciences Research Institute in Berkeley, California during the Fall 2017 semester, by the Clay Institute and the National Science Foundation (Grant No. DMS-1440140).
 Work of F. Santos is also supported by project MTM2017-83750-P of the Spanish Ministry of Science (AEI/FEDER, UE)}

\date{\today}

\begin{abstract}
We construct the first known hollow lattice polytopes of width larger than dimension: a hollow lattice polytope (resp.~a hollow lattice simplex) of dimension $14$ (resp.$~404$) and of width $15$ (resp.$~408$). We also construct a hollow (non-lattice) tetrahedron of width $2+\sqrt2$ and conjecture that this is the maximum width among $3$-dimensional hollow convex bodies.

We show that the maximum lattice width grows (at least) additively with $d$. In particular, the constructions above imply the existence of hollow lattice polytopes (resp. hollow simplices) of arbitrarily large dimension $d$ and width $\simeq 1.14 d$ (resp.$~\simeq 1.01 d$).
\end{abstract}

\maketitle

\section{Introduction}
A convex body $K\subset \R^d$ is called \emph{hollow} or \emph{lattice-free} with respect to a lattice $\Lambda\cong \Z^d$ if $\int(K) \cap \Lambda = \emptyset$.
The \emph{width} of $K$ 
with respect to a linear functional $f\in (\R^d)^*$ is the length of the segment $f(K)$. We denote it $\width(K,f)$. The \emph{lattice width} of $K$ is
\[
\width_\Lambda(K): = \inf_{f\in \Lambda^*\setminus 0} \width(K,f).
\]
We omit $\Lambda$ and write just $\width(K)$ when this creates no ambiguity. We also say \emph{width of $K$} meaning lattice width.

The celebrated \emph{flatness theorem} states that hollow bodies in fixed dimension $d$ have bounded lattice width. That is, for each fixed $d$, the supremum width among all hollow convex bodies in $\R^d$ is a certain constant $w_c(d) < \infty$.
We are also interested in the following specializations of $w_c$. We call $w_p(d)$ the maximum width among all  hollow \emph{lattice $d$-polytopes}, $w_s(d)$ the maximum width among hollow \emph{lattice $d$-simplices} and $w_e(d)$ the maximum width among \emph{empty $d$-simplices}; here a lattice simplex is \emph{empty} if its only lattice points are its vertices. Observe that these specializations take integer values. Obviously,
\[
w_e(d) \le w_s(d) \le w_p(d) \le w_c(d).
\]

Much work has been done in finding upper bounds for $w_c(d)$ (see references, e.g., in the introductions to~\cite{Banaszczyk_etal1999,KannanLovasz1988}).
The current best upper bound is $w_c(d)\in O^*(d^{4/3})$~\cite{Rudelson2000},
where the notation $O(\ )^*$ denotes that a polylog factor is neglected.
Better upper bounds are known for restricted classes of convex bodies.
For example, it is known that the maximum width of hollow (not necessarily lattice) simplices~\cite{Banaszczyk_etal1999} and of centrally symmetric hollow bodies~\cite{Banaszczyk1996} is in $O(n\log n)$.
But work on \emph{lower bounds} is very scarce. To the best of our knowledge, can be summarized as follows:

\begin{itemize}
\item Since the $d$-th dilation of a unimodular $d$-simplex is hollow and has width $d$, $w_s(d) \ge d$.

\item Seb\H{o}~\cite{Sebo1999} showed $w_e(d) \ge d-2$.%
\footnote{As Seb\H{o} points out, for even $d$ the bound can be increased to $d-1$.}

\item Conway and Thompson (see~\cite[Theorem I.9.5]{MilnorHusemoller1973}) showed a lower bound of $\Omega(d)$ for the maximum width of hollow ellipsoids.

\item 
Dash et al.~\cite{Dash_etal2000} (Theorem 3.2 and the paragraphs before it) show that
\[
3.1547\ldots = 2+\frac2{\sqrt3} \le w_s(3) \le w_c(3) \le 1+ \frac2{\sqrt3} +\left(\frac{90}{\pi}\right)^{1/3} = 4.2439\ldots
\]

\item The following exact values are known for small $d$:
\[
\begin{array}{c | rc|rl|rl|rl}
d&\multicolumn{2}{c|}{w_e(d)} & \multicolumn{2}{c|}{w_s(d)} & \multicolumn{2}{c|}{w_p(d)} & \multicolumn{2}{c}{w_c(d)} \\
\hline
1  &  {1} &&  1  &&  1  &&  1 & \\
2  &  {1} &&  2  &&  2  &&  1+\frac2{\sqrt{3}} &\text{\cite{Hurkens1990}} \\
3  &  1&\text{\cite{White1964}}  &    3&\text{\cite{AKW2017}}  &  3&\text{\cite{AKW2017}}  &&    \\
4  & 4&\text{\cite{IglesiasSantos2018}}    &      &&      &&    \\
\end{array}
\]
\end{itemize}

In this paper we establish some new lower bounds, both for specific dimensions (\Cref{sec:dim2,sec:dim3,sec:dim4}) and in the asymptotic sense (\Cref{sec:asymptotic}).

More precisely, in \Cref{sec:dim3} we show that $w_c(3) \ge 2+\sqrt2$:

\begin{theorem}
\label{thm:dim3}
There is a hollow (non-lattice) tetrahedron of width $2+\sqrt2 \simeq 3.4142$.
\end{theorem}

The tetrahedron of \Cref{thm:dim3} is symmetric with respect to the fcc-cubic lattice and has width $2+\sqrt2$ for seven different linear functionals (the three coordinates and four diagonals of the cube). To certify that no integer functional gives smaller width to it we develop in \Cref{sec:certificate} a method which may be of independent interest, based on existence of long piecewise-linear  paths of rational directions.

This tetrahedron maximizes width among a two-parameter family of hollow tetrahedra that contains \emph{two} of the five existing hollow $3$-polytopes of width 3 (see \Cref{thm:family}), in much the same way as  the value of $w_c(2) = 1 + 2/\sqrt{3} \simeq 2.1547$ in the table above is attained by optimizing a perturbation of the second dilation of the unimodular triangle (see details in \Cref{sec:dim2}).
This makes us conjecture that:

\begin{conjecture}
\label{conj:dim3}
The tetrahedron in \Cref{thm:dim3} is the convex $3$-body of largest width; that is,  $w_c(3)=2+\sqrt2$.
\end{conjecture}
 
Similarly, in \Cref{sec:dim2,sec:dim4}, we show $w_p(14)\geq 15$ and $w_s(404) \geq 408$:

\begin{theorem}
\label{thm:explicit}
There is a hollow lattice $14$-polytope of width $15$ and a hollow lattice $404$-simplex of width $408$.
\end{theorem}
We do not know of any hollow lattice $d$-polytope of width larger than $d$ in previous literature.

Our main technical tool, both in \Cref{thm:explicit} and for the asymptotic results, is to use dilated \emph{direct sums} of polytopes and convex bodies.
Let $C_i \subset \R^{d_i}$, $i=1,\dots,m$, be convex bodies containing the origin. Their \emph{direct sum}~\cite{polytope-chapter} (sometimes called \emph{free sum}~\cite{AB2015}) is defined as
\[
C_1\oplus\dots\oplus C_m := \left\{(\lambda_1x_1,\dots,\lambda_m x_m) : x_i \in C_i, \lambda_i \ge 0, \sum_{i=1}^m \lambda_i = 1 \right \} .
\]

For a constant $k\in \R_{\ge 0}$, $kC$ denotes the dilation of $C$ by a factor of $k$. 
For a given lattice polytope or convex body $C$ containing the origin (not necessarily in the interior) let us denote $C^{\oplus m} = \bigoplus_{i=1}^m mC$, the $m$-fold direct sum of $mC$ with itself. The following proposition is a particular case of Theorem~\ref{thm:direct_sum} in Section~\ref{sec:direct_sum}.

\begin{proposition}
\label{prop:directsum}
\begin{enumerate}
    \item $\width(C^{\oplus m})=m\,\width(C)$.
    \item If $C$ is hollow then $C^{\oplus m}$ is hollow.
\end{enumerate}
\end{proposition}

As a consequence we have the following statement, proved in \Cref{sec:asymptotic}.

\begin{theorem}
\label{thm:sup}
Let $w_*: \N \to \R$ denote any of the functions $w_s$, $w_p$, or $w_c$. Then
\[
\lim_{d\to \infty} \frac{w_*(d)}d = \sup_{d\in \N} \frac{w_*(d)}d.
\] 
\end{theorem}

This, in turn, implies our main asymptotic result:

\begin{theorem}
\label{thm:main}
\begin{align*}
 &\lim_{d\to \infty} \frac{w_p(d)}d = \lim_{d\to \infty}  \frac{w_c(d)}d \ge \frac{2 +\sqrt2}{3} = 1.138\dots \\
 &\lim_{d\to \infty}  \frac{w_s(d)}d \ge \frac{102}{101} = 1.0099\dots
\end{align*}
\end{theorem}

\begin{proof}
From ~\Cref{thm:sup}, together with  the explicit lower bounds 
$w_c(3)\ge 2 + \sqrt2$ (\Cref{thm:dim3}) and $w_s(404) \ge 408$ (\Cref{thm:explicit}),
we get
\[
 \lim_{d\to \infty}  \frac{w_c(d)}d \ge  \frac23 +\frac{\sqrt2}3 , \qquad 
 \lim_{d\to \infty}  \frac{w_s(d)}d \ge \frac{102}{101}.
\]
Thus, we only need to show the equality
\[
\sup_{d\in \N} \frac{w_c(d)}d = \sup_{d\in \N} \frac{w_p(d)}d.
\]

The ``$\ge$'' is obvious. For the ``$\le$'', let $C$ be a hollow convex body such that $\width(C)/\dim(C)$ is very close to $\sup_d w_c(d)/d$. We can approximate $C$ arbitrarily by a hollow rational polytope $P$, and choose an integer $m$ such that $mP$ is a lattice polytope. 
By Proposition~\ref{prop:directsum}
we have that $P^{\oplus m}$ is a hollow lattice polytope of dimension $m\dim(C)$ and  
\[
\frac{\width(P^{\oplus m})}{\dim(P^{\oplus m})} = \frac{\width(P)}{\dim(P)} \simeq \frac{\width(C)}{\dim(C)}.
\qedhere
\]
\end{proof}

Another implication of our analysis of direct sums together with the values $w_c(2)=2.1547\dots$ and $w_c(3)\ge 3.4142\dots$ is

\begin{proposition}[\Cref{sec:asymptotic}]
\label{prop:every-d}
For every $d$ we have
\begin{align*}
\label{eq:2and3}
\begin{array}{rl}
w_c(d+1) &\ge \  w_c(d) + 1,\\
w_c(d+2) &\ge \  w_c(d) + 2.1547\dots,\\
w_c(d+3) &\ge \  w_c(d) + 3.4142\dots,
\end{array}
\end{align*}
As a consequence, 
\begin{equation*}
\label{eq:every-d}
\frac{w_c(d)}d \ge \frac12\, 2.1547\dots = 1.0773\dots  \qquad \forall d\ge 2.
\end{equation*}
\end{proposition}

In \Cref{sec:main_e} we study the width of empty simplices.  We do not know whether $\lim_{d\to \infty} \frac{w_e(d)}d$ exists. However, we can prove the following slightly weaker result:

\begin{theorem}[\Cref{sec:main_e}]
\label{thm:main_e}
For every $d,m\in \N$ we have
\[
w_e(dm) \ge (m-3)\,w_e(d).
\]
In particular,
\[
\limsup_{d\to \infty} \frac{w_e(d)}{d} =  \sup_{d\in \N} \frac{w_e(d)}d  \ge 1\\
\]
\end{theorem}

We do not  know whether there is an empty simplex of width larger than its dimension.
Yet, \Cref{thm:main_e} disproves the following guess from \cite[p.~403]{Sebo1999}: 
``it seems to be reasonable to think that the maximum width of an empty integer simplex in $\R^n$ is $n$ + constant'' (unless the constant is zero).

\medskip

We believe our results are a first step towards  the main goal concerning flatness lower bounds, which would be to show that $\sup_d w_*(d)/d = \infty$, at least for $w_c$.

\subsection*{Acknowledgement:} We thank Gennadiy Averkov, Benjamin Nill, and an anonymous referee for useful comments on the first version of this paper.

\section{Hollow direct sums}
\label{sec:direct_sum}

Since we will often be using direct sums of polytopes, let us remind the reader of their combinatorial structure.

\begin{lemma}
\label{lemma:combi-sums}
Let $P=P_1\oplus\dots \oplus P_m$ be a direct sum of polytopes. Then:

\begin{enumerate}
\item If $F_i$ is a face of $P_i$ that does not contain the origin for each $i$ then the join $F_1*\dots* F_m$ of them is a face of $P$ that does not contain the origin of dimension $\sum_i \dim(F_i) + m-1$. All faces of $P$ that do not contain the origin arise in this way. 

\item If $F_i$ is a face of $P_i$ that contains the origin for each $i$ then the direct sum $F_1\oplus\dots\oplus F_m$ of them is a face of $P$ that contains the origin of dimension $\sum_i \dim(F_i)$. All faces of $P$ that contain the origin arise in this way. 
\end{enumerate}

In particular, the non-zero vertices of $P$ are the points of the form $(0,\dots,0,v,$ $0,\dots,0)$, with $v$ a non-zero vertex of the corresponding $P_i$, and $0$ is a vertex of $P$ if and only if it is a vertex of every $P_i$.
\qed
\end{lemma}

Our main technical result is the following theorem. Proposition~\ref{prop:directsum} is the case $C_1=\dots=C_m$ and $k_i=m$ of it. Part (4) of \Cref{thm:direct_sum} is equivalent to Corollary 5.5(a) in \cite{AB2015}.

\begin{theorem}
\label{thm:direct_sum}
Let $C_1,\dots,C_m$ be convex bodies containing the origin and let $k_1,\dots,k_m>0$ be dilation factors. Let
\[
C:= \bigoplus_i k_iC_i = k_1C_1\oplus\dots\oplus k_mC_m.
\]
Then:
\begin{enumerate}

\item If $k_iC_i$ is a lattice polytope for every $i$ then $C$ is a lattice polytope.

\item If $C_i$ is a simplex with a vertex at the origin for every $i$ then $C$ is a simplex with a vertex at the origin.

\item The width of $C$ equals $\min_i \{k_i\width(C_i)\}$.

\item If $C_i$ is hollow for every $i$ and $\sum_i \frac1{k_i} \ge 1$ then $C$ is hollow.
\end{enumerate}
\end{theorem}

\begin{proof}
Part (1) is obvious, from the description of the vertices of direct sums in \Cref{lemma:combi-sums}.
For part (2) let $d_i$ be the dimension of $C_i$. Each $C_i$ has $d_i$ non-zero vertices plus the origin so,
by the same Lemma, $C$ has $d_1 + \dots + d_m$ vertices plus the origin. Since $C$ lives in dimension $d_1 + \dots + d_m$, it must be a simplex.

To prove (3), first note that $\width(k_iC_i) = k_i\width(C_i)$, so we can assume w.l.o.g. $k_i=1$ for all $i$. Let $f_i \in \Lambda_i^*$ be a lattice direction for which $\width(C_i)$ is obtained. Then 

\begin{align*}
    \width(P) \leq & \width(C, (0,\dots,0,f_i, 0, \dots, 0)) \\ 
    = & \width(C_i, f_i) = \width(C_i). 
\end{align*}
This proves that $\width(C) \le \min_i \{\width(C_i)\}$. For the other inequality, given any lattice functional $g=(g_1, \dots, g_m) \in \Lambda^*\setminus\{0\} = \oplus_i \Lambda_i^*\setminus\{0\}$, we want to show that $\width(C,g) \ge \width(C_i)$ for some $i$. For this, let us choose any $i$ with $g_i\ne 0$. Then:

\begin{align*}
      \width(C, g) = & \max\limites_{c, c' \in C} |g^\intercal c - g^\intercal c'|\\
      \geq & \max\limites_{c_i, c_i' \in C_i} |g^\intercal (0, \dots, 0,c_i,0, \dots, 0) - g^\intercal (0, \dots, 0,c_i',0, \dots, 0)| \\
      = & \width(C_i, g_i) \ge \width(C_i) \ge \min_j \width(C_j).
\end{align*}

Finally, to prove part (4), suppose by contradiction that $C$ is not hollow, and let $c \in \int C \cap \Lambda$. Since $c \in \int C$, we can write $c= (\lambda_1 k_1 c_1, \dots, \lambda_m k_m c_m)$ with $c_i \in \int C_i$ and $\lambda_i > 0$ with $\sum \lambda_i = 1$. On the other hand, since $c \in \Lambda$, we know that each $\lambda_i k_i c_i \in \Lambda_i$. Since $C_i$ is hollow and $c_i \in \int C_i$, we have that $\lambda_i k_i > 1$. This implies $\sum \frac1{k_i} < \sum \lambda_i =1$, contradicting our assumption.
\end{proof}

Observe that the assumption that the $C_i$s contain the origin is no loss of generality: lattice polytopes can be translated to have the origin as a vertex; convex bodies can first be enlarged so that they have lattice points in the boundary, then translated. In both cases,  the direct sum $C$ of \Cref{thm:direct_sum} can be constructed using these modified $C_i$s.

\section{A certificate for width}
\label{sec:certificate}
In Sections~\ref{sec:dim2}--\ref{sec:dim4}, we construct explicit examples of polytopes of width larger than their dimension. 
Before that, we show a heuristic method to certify the width of a convex body. This method indirectly takes advantage of the fact that in our examples the width is attained with respect to \emph{several} different functionals. 

By a \emph{rational path} $\Gamma$ in $\R^d$, with respect to a certain lattice $\Lambda$, we mean a concatenation of  segments in rational directions. That is, $\Gamma$ is given as a sequence $p_0, p_1, \dots, p_t$ of points in $\R^d$ such that for every $i$ the vector $p_{i+1}-p_i$ is parallel to a lattice vector. This allows us to define the \emph{lattice length} of each segment $[p_{i}, p_{i+1}]$ as the scalar $\lambda>0$ such that $\frac1{\lambda}(p_{i+1}-p_i)$ is \emph{primitive}, meaning that it is the shortest integer vector in its direction. The \emph{lattice length} of the rational path $\Gamma$ is the sum of the lattice lengths of the individual segments; we denote this by $\length_\Lambda(\Gamma)$. 

We say that a functional $f$ is \emph{strictly increasing} along $\Gamma$ if
\[
f(p_0) < f(p_1) < \dots < f(p_t).
\]
The \emph{open polar cone} of $\Gamma$, denoted $\cone(\Gamma)^\circ$, is the set of functionals $f\in (\R^d)^*$ that are strictly increasing along $\Gamma$.\footnote{The notation $\cone(\Gamma)^\circ$ comes from the fact that this cone is the (open) polar, in the standard sense, of the cone generated by the vectors $p_{i+1}-p_i$, $i=1,\dots,t$.}

\begin{lemma}
\label{lem:path}
Let $P\subset \R^d$ be a convex body. Let $\Gamma$ be a rational path of lattice length $w$ for a certain lattice $\Lambda$, with the first and last points of $\Gamma$ in $P$.
Then any lattice functional in the open polar cone of $\Gamma$ gives width at least $w$ to $P$.
\end{lemma}
\begin{proof}
If $f\in \cone(\Gamma)^\circ$ then
\[
\width(P, f) \ge \length_\Lambda(\Gamma) =w,
\]
since $f$ takes an integer positive value in the primitive vector parallel to each segment of $\Gamma$.
\end{proof}

\begin{remark}
\label{rem:certificate}
As a consequence of the lemma, if  $\Gamma_1,\dots,\Gamma_k$ is a collection of rational paths with end-points in $P$, all of length at least $w$, and
with the property that
\[
\bigcup_{i=1}^k \cone(\Gamma_i)^\circ = \R^d\setminus \{0\},
\]
then the lattice width of $P$ is at least $w$.
\end{remark}

\begin{example}

The necessity for using the \emph{open} polar cone $\cone(\Gamma)^\circ$  and not the closed one in \Cref{lem:path} can be illustrated by considering $P$ to be the square $[0,1]^2$. The two boundary paths between opposite vertices  in $P$ have lattice length two and by \Cref{lem:path} this guarantees that the width of $P$ with respect to any functional in $\Lambda^* \setminus (\langle e_1^*\rangle \cup \langle e_2^*\rangle)$ is at least two. But, of course, the width of $P$ with respect to the functionals $e_1^*$ and $e_2^*$ is $1$, and these two functionals are weakly increasing along the boundary paths.
\end{example}

\section{A hollow lattice $14$-polytope of width 15}
\label{sec:dim2}

Let $A$, $B$ and $C$ be the vertices of an equilateral triangle $\Delta$ in the plane; without loss of generality, $A=(0,0)$, $B=(1,0)$, $C=\left(\frac{1}{2}, \frac{\sqrt3}{2}\right)$. Let $\Lambda$ be the lattice they generate:
\[
\Lambda:= \left\{ (a,b\sqrt3) \in \R^2: 2a, 2b, a+b \in \Z \right\}.
\]

We consider the family $\{T(x,y)\}$ of equilateral triangles circumscribed around $\Delta$, where $(x,y)$ denotes, by convention, the vertex lying between $A$ and $C$. A point $(x,y)$ defines such a triangle if and only if it lies outside $\Delta$
and  along the circle
\[
S_1:=\left\{(x,y): x^2 + \left(y-\frac{\sqrt3}3\right)^2 = \frac{1}{3}\right\}.
\]

It is easy to see that every triangle in the family is hollow. For example, $T(-\frac12, \frac{\sqrt3}2)$ is a hollow lattice triangle of width two, unimodularly equivalent to the second dilation of $\Delta$. 
The triangle $T\left(-\frac{\sqrt3}{3},\frac{\sqrt3}{3}\right)$, pictured in \Cref{fig:widetriangle} (left) maximizes the width of the family and was shown by Hurkens~\cite{Hurkens1990}  to maximize width among all hollow convex $2$-bodies
(see also Averkov and Wagner~\cite{AverkovWagner2012}).

\begin{figure}
\includegraphics[height=4.5cm]{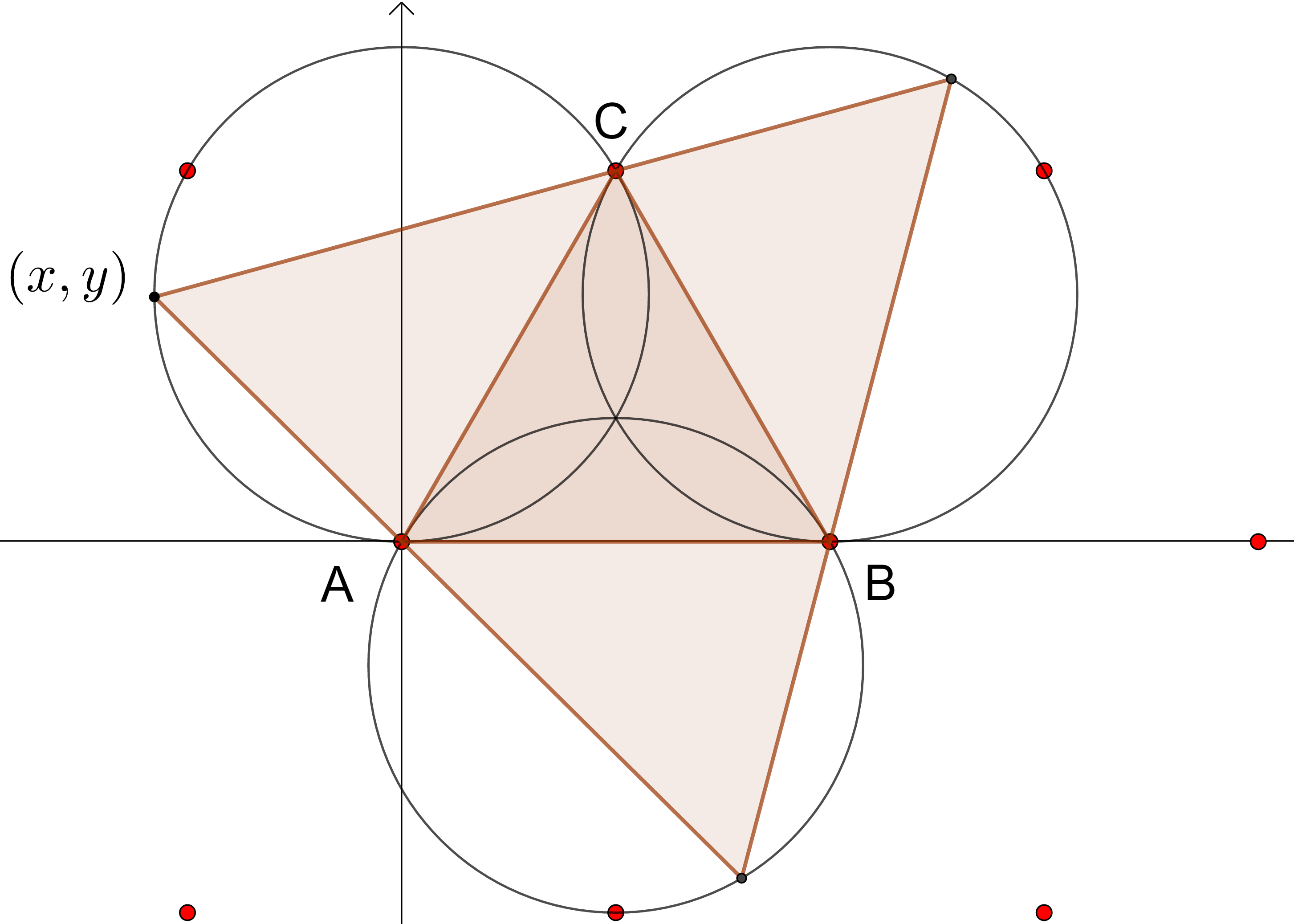}\quad
\includegraphics[height=4.5cm]{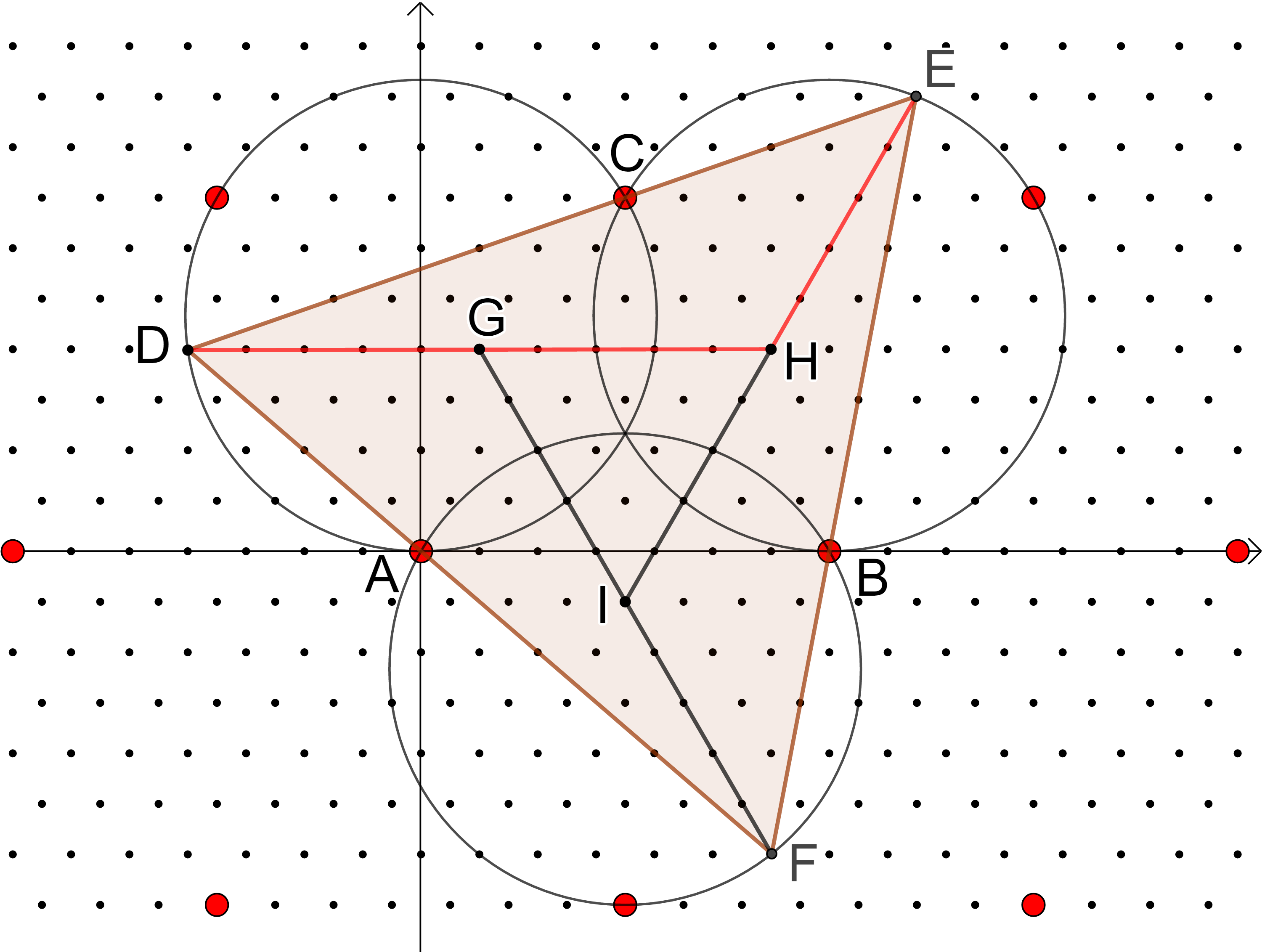}
\caption{
Left: the hollow equilateral triangle $T(x,y)$ circumscribed to the triangle $\Delta$, depending on the position of $(x,y)$ along the circle $S_1$ (Hurkens position,  maximizing width, is shown in the picture). Right: the refinement of the lattice by a factor of seven creates a lattice point in the circle and close to Hurkens position.
}

\label{fig:widetriangle}
\end{figure}

We now consider the seventh refinement $\Lambda':=\frac17\Lambda$ of $\Lambda$. The circle $S_1$ contains, apart from the points $A, B, C$, additional points of $\Lambda'$.
In particular, if we fix $T:=T(D)$ for  the point $D=\left(-\frac47, \frac{2\sqrt3}7\right)\in \Lambda'\cap S_1$
we get a triangle with vertices in $\Lambda'$ and of width close to the maximum, since $D$ is close to Hurken's point $\left(-\frac{\sqrt3}{3},\frac{\sqrt3}{3}\right)$
(See \Cref{fig:widetriangle}, again).
Specifically:
\[
T:= T\left(D\right) = \conv\left(
    \left(-\frac47, \frac{2\sqrt3}7\right) , 
    \left(\frac{17}{14}, \frac{9\sqrt3}{14} \right),
    \left(\frac67, -\frac{3\sqrt3}7\right) 
 \right).
\]

\begin{lemma}
\label{lemma:dim14}
The triangle $T$ defined above is hollow and of width $2+1/7=2.1419$ with respect to $\Lambda$. It is also rational, with its seventh dilation being a lattice triangle.
\end{lemma}

\begin{proof}
It is clear by our construction that $T$ is hollow with respect to $\Lambda$, and since it has its vertices in $\Lambda'$, its seventh dilation is a lattice triangle of $\Lambda$.

We now claim that the width of $T$ in $\Lambda'$ is 15. 
It is easy to check that it has width 15 with respect to the three functionals $f_0$, $f_1$ and $f_2$ that define edges of $\Delta$. 
We call $E$ and $F$ the vertices of $T$ in clockwise order from $D$.
To show that the width of $T$ is at least $15$, we apply \Cref{lem:path} to each of the three paths $DHE$ (drawn in red in \Cref{fig:widetriangle}), $EIF$ and $FGD$.
These paths have lattice length equal to 15. It is easy to see that the polar cones of the paths are $\cone(f_0, f_1)$, $\cone(f_0, f_2)$ and $\cone(f_1, f_2)$, so by \Cref{lem:path}, functionals in the interior of any of these cones give width at least  $15$ to $T$. The only (primitive) functionals not in the open cones are precisely $f_0, f_1$ and $f_2$ which, as said above, yield width $15$. 
\end{proof}

We can now prove the first half of \Cref{thm:explicit}: 

\begin{theorem}
${T}^{\oplus 7}$ is a $14$-dimensional hollow lattice polytope of width $15 $. 
It has $21$ vertices and $2^7+7$ facets ($2^7$ simplices and seven combinatorially of the form segment$\oplus$triangle$^{\oplus 6}$).
\end{theorem}

\begin{proof}
The first claim follows from \Cref{prop:directsum} and \Cref{lemma:dim14}.
${T}^{\oplus 7}$ has $21$ vertices by the description of vertices of direct sums in \Cref{lemma:combi-sums}. The same lemma implies the following description of the facets:

\begin{enumerate}
\item Facets of ${T}^{\oplus 7}$ that do not contain the origin are the joins of edges of $T$ that  do not contain the origin. Since there are two such edges to chose from in $T$ and joins of simplices are simplices, we obtain the $2^7$ stated simplices.

\item Facets of ${T}^{\oplus 7}$ that contain the origin are 
of the form
\[
T \oplus\dots \oplus T \oplus F \oplus T \oplus \dots \oplus T
\]
where $F$ is the edge of $T$ that contains the origin. Since $F$ can be placed anywhere in the sum, we have seven such facets.
\qedhere
\end{enumerate}
\end{proof}

\section{A hollow $3$-simplex of width $3.4142$}
\label{sec:dim3}

Consider the (dilated) face-centered cubic lattice 
\[
\Lambda:=\left\{ (a,b,c) : a,b,c \in 2\Z, a+b+c \in 4\Z\right\},
\]
with dual

\[
\Lambda^* = \left\{(a,b,c) \in \frac14 \Z: a+b, a+c, b+c \in \frac12 \Z^3\right\}.
\]
Here and in what follows we use the standard coordinates in $(\R^3)^*$, so that $(a,b,c)$ denotes the functional $(x,y,z) \mapsto ax+by+cz$.

For the sake of symmetry, all constructions in this section are with respect to the following \emph{affine} lattice, which is a translation of $\Lambda$:
\[
\Lambda_\one:= \Lambda  -(1,1,1) = \left\{ (a,b,c) : a,b,c \in 1+2\Z, a+b+c \in 1 + 4\Z \right\}.
\]
By lattice width with respect to  $\Lambda_\one$ we mean the lattice width with respect to $\Lambda$.

Consider the following lattice tetrahedron (see \Cref{fig:width3}) of width three in $\Lambda_\one$:
\[
\Delta_0:=\conv\{
 \left( 3, 1, 5\right),
 \left( -1, 3, -5\right),
 \left( -3, -1, 5\right),
 \left( 1, -3, -5\right)\}.
\]
$\Delta_0$ is (modulo unimodular transformation) the  hollow $3$-simplex of normalized volume 25 and width $3$ that appears in \cite[Figure 2]{AKW2017} and \cite[Figure 1(h)]{AWW2011}. 
\begin{figure}[htb]
  \centering
\begin{tikzpicture}[mynode/.style={circle, inner sep=1.5pt, fill=red}, scale =
  .8]
  \node[draw,mynode, label=left:{}] () at (-3, 3) {};
  \node[draw,mynode, label=above:{$[-5,-5]$}] () at (-1, 3) {};
  \node[draw,mynode, label=left:{}] () at (1, 3) {};
  \node[draw,mynode, label=left:{}] () at (3, 3) {};
  \node[draw,mynode, label=left:{}] () at (-3, 1) {};
  \node[draw,mynode, label=above:{$[-3,1]$}] () at (-1, 1) {};
  \node[draw,mynode, label=above:{$[-1,3]$}] () at (1, 1) {};
  \node[draw,mynode, label=right:{$[5,5]$}] () at (3, 1) {};
  \node[draw,mynode, label=left:{$[5,5]$}] () at (-3, -1) {};
  \node[draw,mynode, label=below:{$[-1,3]$}] () at (-1, -1) {};
  \node[draw,mynode, label=below:{$[-3,1]$}] () at (1, -1) {};
  \node[draw,mynode, label=left:{}] () at (3, -1) {};
  \node[draw,mynode, label=left:{}] () at (-3, -3) {};
  \node[draw,mynode, label=left:{}] () at (-1, -3) {};
  \node[draw,mynode, label=below:{$[-5,-5]$}] () at (1, -3) {};
  \node[draw,mynode, label=left:{}] () at (3, -3) {};

   \tkzInit[xmax=3.5,ymax=3.5,xmin=-3.5,ymin=-3.5]
        \tkzGrid
        \tkzDrawX
        \tkzDrawY

  \node[draw, label=right:{}] (A) at (3,1) {};  
  \node[draw, label=above:{}] (B) at (-1,3) {}; 
  \node[draw, label=left:{}] (C) at (-3,-1) {};  
  \node[draw, label=below:{}] (D) at (1,-3) {};  
  \draw[black] (A) -- (B) -- (C) -- (D) -- (A) -- (C);
\end{tikzpicture}
\caption{Projection along the $z$-axis of the hollow lattice $3$-simplex $\Delta_0$ of width three in the lattice $\Lambda_\one:=\{(a,b,c) : a,b,c \in 1+2\Z, a+b+c \in 1 + 4\Z\}$. Dots represent (projections of) vertical lattice lines. For those that intersect $\Delta_0$, next to the dot we show the interval of values of $z$ in the intersection. For example, the interval $[-1,3]$ next to the dot with coordinates $(1,1)$ indicates that the points $(1,1,-1)$ and $(1,1,3)$ are in the boundary of $\Delta_0$.
At the four vertices of $\Delta_0$ the interval degenerates to a point.}
\label{fig:width3}
\end{figure}
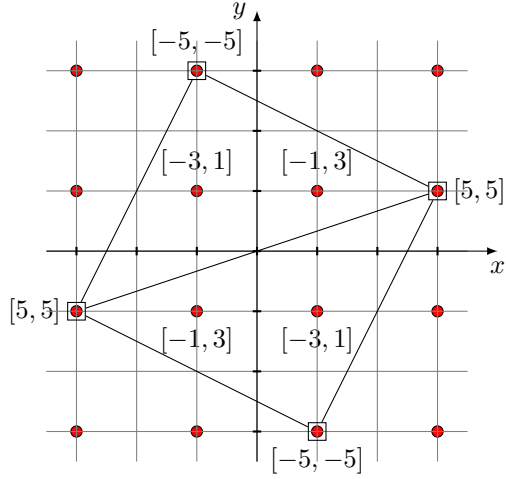
We want to modify $\Delta_0$ to a  non-hollow simplex of larger width, in the spirit of the previous section. We chose this $\Delta_0$ because it achieves its lattice width only with respect to two lattice functionals, namely $x/4$ and $y/4$. This gives a certain freedom to scale down the $z$ coordinate and enlarge the other two, thus increasing the minimum width. We can simultaneously rotate the whole tetrahedron around the $z$ axis.

To formalize this, we consider the family of tetrahedra that share the following properties with $\Delta_0$: they are circumscribed around the unimodular simplex $\conv\{(-1,1,1), (-1,-1,-1), (1,-1,1), (1,1,-1)\}$, and they are invariant under the order four isometry $(x,y,z)\mapsto(-y,x,-z)$. Put differently, for each $(x,y,z)\in \R^3$ we define $\Delta(x,y,z)$ to be the tetrahedron with vertices
\begin{align*}
A=(x,y,z ),\qquad
B=(-y,x, -z ),\qquad
C=(-x,-y,z ),\qquad
D=(y, -x, -z  ).
\end{align*}
We constrain $(x,y,z)$ to satisfy that $(-1,1,1)$, $(-1,-1,-1)$, $ (1,-1,1)$ and $(1,1,-1)$ lie, respectively, in the planes containing $ABC$, $BCD$, $ACD$ and $ABD$. By symmetry, these four conditions are equivalent to one another and
an easy computation shows that they translate to the equality
\begin{align}
\label{eq:family}
z= \frac{x^2+y^2}{x^2+y^2 -2x-2y}.
\end{align}

In the rest of this section we show the following, which implies \Cref{thm:dim3}:

\begin{theorem}
\label{thm:family}
Let $(x,y,z) \in \R^3$ be a point satisfying the constraint of \Cref{eq:family}.
Then, the width of $\Delta(x,y,z)$ with respect to $\Lambda$ is at most  $2+\sqrt{2}$, with equality if and only if 
\[
(x,y,z) \in \left\{\ \left(\, 2+ {\sqrt2}, \  {\sqrt2}, \ 2+ {\sqrt2}\, \right),\ \ \left(\, {\sqrt2}, \ 2+ {\sqrt2}, \ 2+ {\sqrt2}\, \right)\ \right\}.
\]
\end{theorem}

\begin{proof}[Proof of the upper bound in \Cref{thm:family}]
Let us consider the functionals $(\frac12,0,0)$, $(0, \frac12,0)$, and $(0,0,\frac12)$, which are in $\Lambda^*$.
The width of $\Delta(x,y,z)$ with respect to first two equals $\max\{|x|,|y|\}$, and with respect to the third equals $|z|$. We are going to show that  whenever $\max\{|x|,|y|\} \ge 2 + \sqrt2$ we have $|z|\le 2+\sqrt2$.
Let
\[
f(x,y) := \frac{x^2+y^2}{x^2+y^2 -2x-2y}
\]
be the function giving $z$ in terms of $x$ and $y$. The assumption $\max\{|x|,|y|\} \ge 2 + \sqrt2$ implies that the denominator of $f$ is positive,
since it is only negative (or zero) inside (or on) the circle with center $(1,1)$ and radius $\sqrt2$. The numerator is also obviously positive, and thus $z$ is positive. The equation
\[
f(x,y) = \frac{x^2+y^2}{x^2+y^2 -2x-2y} = 2+\sqrt2
\]
defines again a circle, with center $(\sqrt2, \sqrt2)$ and radius $2$. Outside the circle $z$ is smaller than $2+\sqrt2$ and inside the circle at least one of $|x|$ and $|y|$ is. 

\end{proof}

Thus, for the rest of the section we fix $\Delta = \Delta(2+ {\sqrt2}, \ \  {\sqrt2}, \ \ 2+ {\sqrt2})$, which has the following vertices and is depicted in \Cref{fig:width3.4}:
\begin{align*}
A&=\big(\ 2+ {\sqrt2}, \ \  {\sqrt2}, \ \ 2+ {\sqrt2} ),&
B&=\big(- {\sqrt2}, \ \ 2+{\sqrt2}, \ \ -2- {\sqrt2} ),\\
C&=\big(-2- {\sqrt2}, \ \ -{\sqrt2}, \ \ 2+ {\sqrt2} ),&
D&=\big(\ {\sqrt2}, \ \ -2-{\sqrt2}, \ \ -2- {\sqrt2} ).
\end{align*}

\begin{figure}[htb]
  \centering
\begin{tikzpicture}[mynode/.style={circle, inner sep=1.5pt, fill=red}, scale =  .8]
  \node[draw,mynode, label=left:{}] () at (-3, 3) {};
  \node[draw,mynode, label=below:{$[-3,-\alpha]$}] () at (-1, 3) {};
  \node[draw,mynode, label=left:{}] () at (1, 3) {};
  \node[draw,mynode, label=left:{}] () at (3, 3) {};
  \node[draw,mynode, label=left:{}] () at (-3, 1) {};
  \node[draw,mynode, label=below:{$[-\alpha, 1]$}] () at (-1, 1) {};
  \node[draw,mynode, label=above:{$[-1,\alpha]$}] () at (1, 1) {};
  \node[draw,mynode, label=below:{$[\alpha, 3]$}] () at (3, 1) {};
  \node[draw,mynode, label=above:{$[\alpha, 3]$}] () at (-3, -1) {};
  \node[draw,mynode, label=below:{$[-1,\alpha]$}] () at (-1, -1) {};
  \node[draw,mynode, label=above:{$[-\alpha, 1]$}] () at (1, -1) {};
  \node[draw,mynode, label=left:{}] () at (3, -1) {};
  \node[draw,mynode, label=left:{}] () at (-3, -3) {};
  \node[draw,mynode, label=left:{}] () at (-1, -3) {};
  \node[draw,mynode, label=above:{$[-3,-\alpha]$}] () at (1, -3) {};
  \node[draw,mynode, label=left:{}] () at (3, -3) {};

   \tkzInit[xmax=3.5,ymax=3.5,xmin=-3.5,ymin=-3.5]
        \tkzGrid
        \tkzDrawX
        \tkzDrawY

  \node[draw, label=right:{}] (A) at (3.4142, 1.4142) {$A$};  
  \node[draw, label=above:{}] (B) at (-1.4142, 3.4142) {$B$}; 
  \node[draw, label=left:{}] (C) at (-3.4142, -1.4142) {$C$};  
  \node[draw, label=below:{}] (D) at (1.4142, -3.4142) {$D$};  
  \draw[black] (A) -- (B) -- (C) -- (D) -- (A) -- (C);
\end{tikzpicture}

\caption{The hollow $3$-simplex $\Delta$ of width $2+\sqrt2$, drawn with the same conventions as in \Cref{fig:width3}. We abbreviate $\alpha=1+\sqrt2$.}
\label{fig:width3.4}
\end{figure}
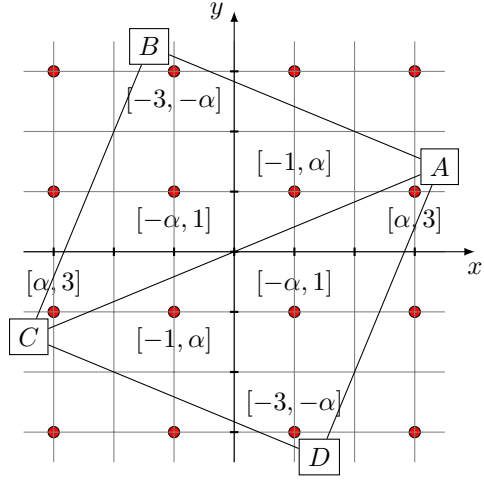

Observe that the width of $\Delta$ with respect to the following 14 lattice functionals equals $2+\sqrt2$:
\begin{equation}
\label{eq:width-tetrahedron}
\begin{array}{cccc}
\pm\frac12(1,0,0), &\pm\frac12(0,1,0), &\pm\frac12(0,0,1), 
&\frac14(\pm1,\pm1,\pm1). \\
\end{array}
\end{equation}
We now prove that this is the width of $\Delta$.

\begin{proof}[Proof of the equality in \Cref{thm:family}]

In \Cref{fig:width3.4} we have written, next to each vertical lattice line $\ell=(x_0,y_0) \times \R$ intersected by $\Delta$, the interval $\{z\in \R: (x_0,y_0,z)\in \Delta\}$. Hollowness follows from this information, since the intervals do not contain points of $\Lambda$ in their interior.
To check correctness of these computations observe that the facet-defining inequality for triangle $ABC$ is
\[
z \le \frac{x-y}{\sqrt2} - y + 2+\sqrt2.
\]
Plugging in the coordinates $(x_0,y_0) \in \{(-3,-1), (-1,1), (-1,3), (1,1)\}$ of the four vertical lines meeting the triangle $ABC$ we get that the highest points of $\Delta$ on each are indeed
\[
(-3,-1,{\bf 3}), \ (-1,1,{\bf 1}), \ (-1,3,{\bf -1-\sqrt2}), \ (1,1, {\bf 1+\sqrt2}).
\]
The rest of upper and lower bounds for the intervals in \Cref{fig:width3.4} follow by symmetry.

To show that the width is at least $2+\sqrt2$ we apply \Cref{lem:path} to various paths. 
For example, the expression
\[
D = C + \left(\frac12 +\frac{\sqrt2}2\right)(4,0,0) + \frac12 (0,-4,0) + \left(1+\frac{\sqrt2}2\right) (0,0,-4).
\]
gives a rational path from vertex $C$ to vertex $D$ with directions $(4,0,0)$,  $(0,-4,0)$ and $(0,0,-4)$ and of length
\[
\left(\frac12 +\frac{\sqrt2}2\right) + \frac12 + \left(1+\frac{\sqrt2}2\right)= 2+\sqrt2.
\]
The open polar cone of this rational path is the octant $\{(a,b,c)\in (\R^3)^*: a>0, b<0, c<0\}$, so  all lattice functionals in the interior of the octant give width at least $2+\sqrt2$ to $\Delta$. The same path in reverse implies the same for the opposite octant, and the  symmetry of order 4 in $\Delta$ implies it for the eight open octants.

We now define a second family of paths whose open polar cones are the connected components of $(\R^3)^*\setminus \{(a,b,c): a=\pm b\}$. (Observe that these are non-pointed cones).
The first one goes from $B$ to $D$ based on the equality
\[
D= B + (-2,-2,0) + (1+\sqrt2)(2,-2,0).
\]
Its length is $1+ (1+\sqrt2)= 2 + \sqrt2$ and its open polar cone is
\[
\{(a,b,c) : a+b < 0, a-b >0\}.
\]
Again, symmetry of $\Delta$ gives paths for the other three connected components of $(\R^3)^*\setminus \{(a,b,c): a=\pm b\}$.

Together, these two sets of paths show width $\ge 2 +\sqrt2$ for all lattice functionals 
except for 
the integer multiples of $\frac12(0,0,1)$, $\frac12(1,1,0)$ and $\frac12(1,-1,0)$. These three give widths $2+\sqrt2$, $2+2\sqrt2$ and $2+2\sqrt2$ to $\Delta$, respectively.
\end{proof}

\begin{remark}
The family of tetrahedra $\Delta(x,y,z)$ also contains 
\[
\Delta(3,3,3) = \conv\{
(3,3,3), 
(3,-3,-3), 
(-3,3,-3), 
(-3,3,-3)\},
\]
which is the  third dilation of a unimodular simplex.
In this sense, $\Delta(x,y,z)$ is a common generalization of two of the three existing lattice tetrahedra of maximal width~\cite{AKW2017}. This is further motivation for~\Cref{conj:dim3}.
\end{remark}

\section{A hollow lattice $404$-simplex of width 408}
\label{sec:dim4}

We now want to construct a lattice \emph{simplex} of width larger than its dimension. 
To do this via \Cref{thm:direct_sum}, we need a rational hollow simplex with the origin as a vertex and of width larger than its dimension, which can be found in dimension four. We do not know whether one exists in dimension three.

\begin{lemma}
\label{lemma:dim404}
There is a rational hollow $4$-simplex $S$ of width $4+4/101$ and with a lattice vertex whose $101$-th dilation is a lattice simplex.
\end{lemma}

\begin{proof}
It is known that the following lattice $4$-simplex is \emph{empty}, that is, it has no lattice points other than its vertices, and it has width four (\cite{HaaseZiegler2000,IglesiasSantos2018}):
\[
S_0:= \conv\{(0,0,0,0), (1,0,0,0), (0,1,0,0), (0,0,1,0),  (6, 14, 17, 101)\}.
\]
Observe that the facet of $S_0$ opposite the origin lies in the hyperplane $101x_1 +101x_2 +101x_3 - 36 x_4 = 101$. Since $101$ is coprime with $36$, dilating $S_0$ by a factor of $102/101$ gives a hollow simplex $S$: apart of the five vertices of $S_0$ (which lie in the boundary of $S$) all other lattice points must be in the facet-defining hyperplane $101x_1 +101x_2 +101x_3 - 36 x_4 = 102$.
\end{proof}

Applying \Cref{prop:directsum} to the hollow simplex $S$, we obtain that $S^{\oplus 101}$ is a $404$-dimensional lattice simplex of width $408$. This proves the second half of \Cref{thm:explicit}. 

\begin{remark}
Any dilation of $S_0$ by a factor strictly greater than $102/101$ is not hollow anymore, since the point 
{\small
\[
(1,2,3,14) = \frac{17}{101} (1,0,0,0)+\frac{6}{101} (0,1,0,0)+\frac{65}{101} (0,0,1,0)+\frac{14}{101} (6,14,17,101)
\]
}
lies in the relative interior of the facet of $S$ opposite the origin.
\end{remark}

\section{General lower bounds}
\label{sec:asymptotic}
In this section, we apply \Cref{thm:direct_sum} to the explicit examples from Sections~\ref{sec:dim2}--\ref{sec:dim4} to 
obtain lower bounds for $w_c(d)$, $w_p(d)$ and $w_s(d)$ in general dimension $d$. In
particular, we prove~\Cref{thm:sup} and~\Cref{prop:every-d}.

\begin{corollary}
\label{cor:superadditive}
For $w_*=w_c$, $w_p$ or $w_s$ we have that
\[
w_*(m d) \ge m\, w_*(d), \quad \forall m\in \N.
\]
For $w_c$ we have the more general result
\begin{equation}
\label{eq:superadditive}
{w_c(d_1+\dots+d_m)} \ge  w_c(d_1) +\dots + w_c(d_m), \quad \forall d_1,\dots,d_m \in \N.
\end{equation}
\end{corollary}

\begin{proof}
For the first inequality, let $K$ be a hollow convex $d$-body (resp., a lattice $d$-polytope, a lattice $d$-simplex) achieving $w_c(d)$ (resp. $w_p(d)$, $w_s(d)$ and, in the case of a lattice polytope, assume without loss of generality that the origin is a vertex of $K$). Then, apply Theorem~\ref{thm:direct_sum} with $C_i=K$ and $k_i=m$ for all $i$. This gives a $dm$-dimensional hollow convex body (resp., a lattice polytope, a lattice simplex) of width $m\,w_*(d)$.

For the case of $w_c$ we have more freedom, since we do not need the $k_i$s to be integers.
Thus, if for each $i=1,\dots,m$ we let the $C_i$s in Theorem~\ref{thm:direct_sum} be hollow convex bodies of width $w_c(d_i)$ and we take $k_i = (\sum_j w_c(d_j)) / w_c(d_i)$ for each $i$, we obtain a hollow convex body $C$ of width $\sum_j w_c(d_j)$ and dimension $\sum_j d_j$.
\end{proof}

Inequality~\eqref{eq:superadditive} implies that $w_c$ is strictly increasing. For $w_p$ and $w_s$ we can only prove weak monotonicity:

\begin{corollary}
\label{cor:monotone}
\[
w_p(d+1) \ge w_p(d) ,
\quad
w_s(d+1) \ge w_s(d) .
\] 
\end{corollary}

\begin{proof}
Let $C_1=[0,1]$ and let $C_2$ be a lattice polytope (resp. a hollow simplex) of dimension $d$ and with $\width(C)=w_*(d)$. Apply Theorem~\ref{thm:direct_sum} with
$k_1 > w_*(d)$ and $k_2 =1$.
\end{proof}

\begin{question}
Are $w_p$, $w_s$ or $w_e$ strictly increasing? Since these $w_*$ take only integer values,
strict monotonicity is equivalent to the inequality
\[
w_*(d+1) \ge w_*(d) + 1.
\]
(For $w_e$ even non-strict monotonicity is unclear, due to its more arithmetic nature).
\end{question}

We can now prove~\Cref{thm:sup} and~\Cref{prop:every-d}:

\begin{proof}[Proof of ~\Cref{thm:sup}]
By Corollaries~\ref{cor:superadditive} and~\ref{cor:monotone}, the three sequences $w_c(d)$, $w_p(d)$ and $w_s(d)$ satisfy the conditions of the following elementary statement: 

If a sequence  $(a_d)_{d\in \N}$ satisfies
$a_{d+1}\ge a_d$ and
$a_{md}\ge m\,a_d$ 
$\forall d,m\in \N$,
then
\[
\lim_{d\to\infty} \frac{a_d}d = \sup_{d\in \N} \frac{a_d}d.
\qedhere
\]
\end{proof}

\begin{proof}[Proof of \Cref{prop:every-d}]
The inequalities 
\begin{align*}
\begin{array}{rl}
w_c(d+1) &\ge \  w_c(d) + 1,\\
w_c(d+2) &\ge \  w_c(d) + 1+\frac{2}{\sqrt3},\\
w_c(d+3) &\ge \  w_c(d) + 2+ \sqrt2,
\end{array}
\end{align*}
follow from applying \Cref{eq:superadditive} of \Cref{cor:superadditive} with $w_c(1)=1$
$w_c(2)=1+\frac{2}{\sqrt3}$~\cite{Hurkens1990} and $w_c(3)\ge2+\sqrt2$~(\Cref{sec:dim3}).

Any integer $d \geq 2$ can be written as $d=2a+3b$ for some nonnegative integers $a, b$. Then for all $d \geq 2$, the inequalities above yield
\begin{align*}
w_c(d)  \geq w_c(2)a+w_c(3)b  
&\geq \left(1+\frac{2}{\sqrt3}\right)a + \left(2+\sqrt2\right)b \\
&\geq \left(\frac12+\frac1{\sqrt3}\right) \left(2a+3b\right)=\left(\frac12+\frac1{\sqrt3}\right) d.
\qedhere
\end{align*}
\end{proof}

\section{Lower bound for empty simplices}
\label{sec:main_e}
To prove the asymptotic lower bound of Theorem~\ref{thm:main_e}, we use the following lemma:

\begin{lemma}
\label{lemma:Sebo}
Let $P=\conv(0, v_1,\dots, v_d) \subset \R^d$ be an empty $d$-simplex of width $w$ and let $m\ge 2$ be an integer.
For each $i\in [d]$ and $j\in [m]$, let
\[
v_i^{(j)} := 0\oplus \dots\oplus v_i  \oplus \dots\oplus 0 \in \R^{md}
\]
with $v_i$ in the $j$-th summand, and define
\[
w_i^{(j)} := (m-2) v^{(j)}_i + v^{(j+1)}_{i+1} \in \R^{md},
\]
with $i$ taken modulo $d$ and $j$ modulo $m$. Let
\[
P_m := \conv\left(\big\{0\big\} \cup \left\{ w_i^{(j)} : (i,j)\in [d]\times[m]\right\}\right).
\]
Then: {\rm (1)} $\width(P_m) \ge (m-3) w$; and {\rm (2)} $P_m$ is empty.
\end{lemma}

\begin{proof}
Observe that $P_m$ is contained in $P^{\oplus m}=\bigoplus_{j=1}^m mP$ and tries to
approximate it: the vertices of $P^{\oplus m}$ are $0$ and $\{mv_i^{(j)}: (i,j)\in [d]\times[m]\}$, and the vertices $w_i^{(j)}$ of $P_m$ are close to them.

To prove (1), let $f=f_1\oplus \dots \oplus f_m$ be an integer functional. Assume without loss of generality that 
\[
\max_j\{\width(P, f_j)\} = \width(P, f_1).
\]
Let us denote $v_0=0$ and let 
$i^+,i^-\in \{0,\dots, d\}$ be indices such that $f_1(v_{i^+})$ and $f_1(v_{i^-})$ are the maximum and minimum 
values of $f_1$ on $P$, respectively.
Then,
\begin{align*}
|f(w_{i^+}^{(1)} - w_{i^-}^{(1)}) | 
&\ge |f_1((m-2) v_{i^+} - (m-2) v_{i^-}) | - |f_2( v_{i^++1} - v_{i^-+1}) | \\
&\ge
(m-2) |f_1( v_{i^+} - v_{i^-}) | - |f_1( v_{i^+} - v_{i^-}) | \\
&= 
(m-3) |f_1( v_{i^+} - v_{i^-}) | \ge (m-3)w.
\end{align*}

For part (2), to search for a contradiction assume $P_m$ is not empty. Let $z\in P_m \cap \Z^{dm}$ be an integer point different from $0$ and from the $w_i^{(j)}$s.
We can then write $z$ as a convex combination of the vertices of $P_m$. That is:
\begin{equation}
\label{eq:lambda}
z = 
\sum_{j=1}^m \sum_{i=1}^d \lambda_i^{(j)} w_i^{(j)}
=
\sum_{j=1}^m \sum_{i=1}^d ((m-2)\lambda_i^{(j)} + \lambda_{i-1}^{(j-1)})  v^{(j)}_i,
\end{equation}
with $\lambda_i^{(j)}\ge 0$ and $\sum_{i,j} \lambda_i^{(j)}\le 1$.

But since $P_m \subset P^{\oplus m}$, we can also write
\begin{equation}
\label{eq:mu}
z= \mu_1 z_1 \oplus \dots \oplus \mu_m z_m,
\end{equation}
with each $z_j\in P$, $\mu_j z_j \in \Z^d$, $\mu_j\ge 0$ and $\sum_j \mu_j \le m$. 
Comparing Equations~\eqref{eq:lambda} and \eqref{eq:mu} we obtain
\begin{equation}
\label{eq:mu_lambda}
\mu_j z_j 
=\sum_{i=1}^d ((m-2) \lambda_i^{(j)} + \lambda_{i-1}^{(j-1)} )v_{i}.
\end{equation}

\emph{Claim: $\sum_i \lambda_i^{(j)} \ne 0$ for every $j$}. Indeed, if there is a $j$ where this sum is zero, assume without loss of generality that $\sum_i \lambda_i^{(j-1)} \ne 0$. Then  Equation \eqref{eq:mu_lambda} gives
\[
\mu_j z_j 
=\sum_{i=1}^d \lambda_{i-1}^{(j-1)} v_{i},
\]
which is a nonzero point in $P$. Since $P$ is empty and $\mu_j z_j \in \Z^d$, we conclude that one of the $\lambda_{i}^{(j-1)}$s equals 1, so that $z = w_i^{(j-1)}$, a contradiction because $z$ was assumed not to be a vertex of $P_m$.

From the claim and Equation \eqref{eq:mu_lambda} it  follows that $\mu_j z_j\ne 0$ for all $j$. In order for 
$\mu_j z_j$ to be a lattice point we need $\mu_j\ge 1$ (because $0<\mu_j<1$ implies $\mu_j z_j$ to be a lattice point in $P$ but not a vertex of $P$, which is not possible). Since on the other hand $\sum_j\mu_j\le m$, we conclude that
$\mu_j = 1$ for every $j$.
This implies that every $z_j$ is a non-zero lattice point of $P$; that is, for each $j$ there is an $i_j$ such that 
$z_j =v_{i_j}$. Equation~\eqref{eq:mu_lambda} now becomes
\[
v_{i_j} 
=\sum_{i=1}^d ((m-2) \lambda_i^{(j)} + \lambda_{i-1}^{(j-1)} )v_{i}.
\]
Since the $v_i$s are independent, we have 
\[
1 = (m-2) \lambda_{i_j}^{(j)} + \lambda_{i_j-1}^{(j-1)}, \qquad \forall j.
\]
Summing over $j$ we get the contradiction
\[
m =\sum_j  (m-2) \lambda_{i_j}^{(j)} + \sum_j  \lambda_{i_j-1}^{(j-1)} \le (m-2) + 1 = m-1.
\qedhere
\]
\end{proof}

\begin{remark}
\Cref{lemma:Sebo} and its proof generalize Seb\H{o}'s construction of empty $m$-simplices of width $m-2$  \cite{Sebo1999}. Indeed, letting $P=[0,1]$ our lemma gives an empty $m$-simplex of width (at least) $m-3$. Seb\H{o}'s $m-2$ is obtained with an additional argument that works for $[0,1]$ but not (as far as we can see) for an arbitrary $P$.
\end{remark}

\begin{proof}[Proof of Theorem~\ref{thm:main_e}]
Let $P$ be an empty $d$-simplex of maximum width; that is, with $\width(P)=w_e(d)$.
Applying Lemma~\ref{lemma:Sebo} to  $P$ we obtain a sequence $(P_m)_{m\in \N}$ of empty $md$-simplices of width $(m-3) w_e(d)$, which implies
$
w_e(dm) \ge (m-3)\,w_e(d).
$

From this fact, combined with $w_e(1)=1$, we obtain
\[
\limsup_{d\to \infty} \frac{w_e(d)}{d} =  \sup_{d\in \N} \frac{w_e(d)}d  \ge 1.\\
\qedhere
\]
\end{proof}

\end{document}